\documentclass[10pt]{amsart}%
\usepackage{amsmath}
\usepackage{amsfonts}
\usepackage{hyperref}
\usepackage{amssymb}
\usepackage{graphicx}%
\usepackage[capitalize]{cleveref}
\usepackage{color}
\usepackage{verbatim}
\usepackage{tikz-cd}
\usepackage{comment}
\usepackage{bbm}
\usepackage[normalem]{ulem}
\providecommand{\U}[1]{\protect\rule{.1in}{.1in}}
\newtheorem{theorem}{Theorem}[section]

\newtheorem{corollary}[theorem]{Corollary}

\newtheorem{definition}[theorem]{Definition}

\newtheorem{lemma}[theorem]{Lemma}

\newtheorem{question}[theorem]{Question}
\newtheorem{proposition}[theorem]{Proposition}

\def\R{{\mathbb R}}
\def\N{\mathbb{N}}
\def\Ntwo{\mathbb{N}_{\geq 2}}
\def\Z{{\mathbb Z}}

\begin{document}
\title{Coincidence rank and multivariate equicontinuity}
\author{Felipe García-Ramos}
\author{Irma León-Torres}

\address{F. García-Ramos, Physics Institute, Universidad Autónoma de San Luis Potos\'i, M\'exico\\
Faculty of Mathematics and Computer Science, Uniwersytet Jagielloński, Poland.}
\email{fgramos@conahcyt.mx *corresponding author}  
\address{I. León-Torres, Physics Institute, Universidad Aut\'onoma de San Luis Potos\'i, M\'exico.}
\email{yooirma@gmail.com}
\begin{abstract}
  The coincidence rank, introduced by Barge and Kwapisz, measures the regularity of the maximal equicontinuous factor of minimal dynamical systems. We provide a characterization of the finiteness of coincidence rank using a multivariate notion of equicontinuity. 
\end{abstract}
\maketitle


\section{Introduction}
In topological dynamics, minimal equicontinuous systems represent the most structured and ordered types of systems. These systems can be modeled as rotations on compact abelian groups. However, most dynamical systems, even topological models of quasicrystals, are not of this form. A key question in such cases is how closely these systems approximate equicontinuous behavior. One common approach is to study the \textit{maximal equicontinuous factor}, which is the largest dynamical projection of the system that retains equicontinuity. This factor provides insight into the regularity and structure of the original system. For chaotic systems, like mixing systems, the maximal equicontinuous factor reduces to a single point. In contrast, systems with more regular behavior exhibit a richer connection to their maximal equicontinuous factor, where properties such as having finite fibers become relevant.

 In the study of aperiodic order, there are different measurements of how close a system resembles its maximal equicontinuous factor. The \textit{maximal rank} is the supremum of the cardinality of the fibers of the maximal equicontinuous factor map. Another notion, introduced by Barge and Kwapisz\cite{barge2006geometric}, is the \textit{coincidence rank}, that is, the maximal number of distinct points in a fiber, of the maximal equicontinuous factor map, that are not similarly behaved (not proximal). This concept was introduced to study substitution dynamical systems, but the concept has been explored for more general minimal systems and actions in \cite{barge2013proximality,barge2012homological,kellendonk2022ellis,baker2006geometric}. In particular, consequences of having finite coincidence rank were studied in \cite[Subsection 4.2.3]{Aujogue2015}. 

In addition to studying equicontinuity via the maximal equicontinuous factor, other approaches focus on regularity within the system itself. A typical approach would be to define conjugate-invariant properties that resemble equicontinuity in some way. Such notions involve tuples of points that remain close on specific subsets of the acting group.  
 A natural question arises: can one identify weaker forms of equicontinuity that appear within the system using the regularity of the maximal equicontinuous factor and vice-versa. Several studies have explored these connections using concepts such as syndetic equicontinuity \cite{huang2018analogues} and mean equicontinuity and its variants \cite{downarowiczglasner,garcia2017weak,garcia2021mean,li2015mean}. 

In this paper, we extend this line of investigation by introducing two multivariate notions of equicontinuity: $m$-equicontinuity and $m$-convering equicontinuity and we use them to characterize finite rank. 
\begin{theorem}
\label{teorema1.1}
    Let $G$ be a $\sigma$-compact abelian group, $(X,G)$ a minimal continuous action and $m\in \N$. We have that 
    \begin{itemize}
        \item  $(X,G)$ is $m$-equicontinuous if and only if the maximal rank is bounded by $m$.
        \item  $(X,G)$ is cover $m$-equicontinuous if and only if the coincidence rank is bounded by $m$.
    \end{itemize}
   
\end{theorem}


In recent papers, Breitenbücher, Haupt, and Jäger and Haupt introduced different multivariate forms of equicontinuity, which they applied to the study of finite-to-one topomorphic extensions \cite{breitenbucher2024multivariate,haupt2025multivariate}.  While their results are not directly related to ours, they contribute to the same general line of research. 

The study of non-trivial finite ranks and multivariate notions is not only of abstract interest but is also supported by several natural examples that exhibit these properties. Classical examples include the Thue-Morse substitution (and its generalizations) and some Toeplitz subshifts. Additionally, the existence of examples with finite coincidence rank but infinite minimal rank can be stated using skew-product techniques developed by Glasner and Weiss \cite{glasner1979construction}. Explicit smooth constructions of this type were presented by Haupt and Jäger \cite{HAUPT_JAGER_2024}. We will provide detailed references for these examples in Section \ref{sec:examples}.

Our results are related to the study of multivariate regional proximality and sensitivity \cite{auslander2004group,Shao2008,Huang2011,zou2017stronger,liu2019auslander,Lituple2022}. In occasions the results are somewhat natural adaptations. We will mention such connections throughout the paper. \\

\textbf{Ackowledgments:} The authors would like to thank Tobias Jägger and Sebastián Donoso for conversations and comments that improved the paper. We would also like to thank Song Shao for valuable references. I. León-Torres would like to thank Alicia Santiago Santos for introducing her to the topic of sensitivity. 

I. León-Torres was supported by a SECIHTI (CONAHCyT) fellowship for her PhD studies, along with the complementary support for Indigenous women. She also extends her gratitude to the Jagiellonian University and the University of Chile for hosting research visits.

F.~García-Ramos was supported by a SECIHTI (CONAHCyT) Basic Science grant, the grant U1U/W16/NO/01.03 of the Strategic Excellence Initiative program of
the Jagiellonian University, and the grant K/NCN/000198 of the NCN.

\section{Preliminaries}

The collection of positive integers will be denoted by $\N$ and the collection of integers larger or equal than $2$ with $\N_{\geq 2}$.

Throughout this paper $G$ represents a locally compact topological group with identity $e$, and $X$ a compact metrizable space with compatible metric $d$. 

Given $x\in X$ and $\varepsilon>0$ we define $B_{\varepsilon}(x)=\{y\in X: d(x,y)<\varepsilon\}$.

A \textbf{group action} of $G$ on $X$ is a map $\alpha:G\times X\to X$, such that \[\alpha(g,\alpha(h,x))=\alpha(gh,x)\] and $\alpha(e,x)=x$ for all $x\in X$ and $g,h\in G$. We say that the triplet $(X,G,\alpha)$ is a \textbf{continuous action} if $\alpha:G\times X\to X$ is a group action that is continuous. For convenience, we will simply use the pair $(X,G)$ to represent a continuous action, and we will denote the image $\alpha(g,x)$ as $gx$. If $Y\subset X$ then $gY=\{gy:y\in Y\}$.\\

We say $(X,T)$ is a \textbf{topological dynamical system} and $T\colon X \to X$ is a homeomorphism. There is a correspondence between continuous actions of $(X,\Z)$ and topological dynamical systems $(X,T)$. \\

Let $(X,G)$ be a continuous action and $x\in X$. The \textbf{orbit of $x$} is the set $Gx=\{gx:g\in G\}$. We say that $(X,G)$ is \textbf{minimal} if $Gx$ is dense in $X$ for all $x\in X$. Actually, $(X,G)$ is minimal if and only if for every non-empty $G$-invariant ($gY=Y \; \forall g\in G$) closed subset $Y\subset X$ we have that $Y=X$.  \\

We say that a pair $(x_1,x_2)\in X\times X$ is \textbf{proximal} if for any $\varepsilon>0$,  there exists $g\in G$ such that $d(gx_1,gx_2)<\varepsilon$. We denote the set of all proximal pairs of $(X,G)$ by $P(X,G)$. 

Let $m\in \Ntwo$. A tuple $(x_1,\ldots,x_m)\in X^m$ is called \textbf{$m$-regionally proximal} if for each $\varepsilon>0$, there exist $x'_1,\ldots,x'_m\in X$ such that $d(x_i,x'_i)<\varepsilon$ for all $i\in\{1,\ldots,m\}$, and there exists $g\in G$ such that $d(gx'_i,gx'_j)<\varepsilon$ for all $i,j\in \{1,\ldots,m\}$. We denote by $Q_m(X,G)$ the collection of all $m$-regionally proximal tuples of $(X,G)$. The pairs in $Q_2(X,G)$ correspond to the classical regional proximal pairs.\\ 

The following result was proved in \cite[Theorem 8]{auslander2004group}.
\begin{theorem}[Auslander]
\label{nregionally}
    Let $G$ be an abelian group, $(X,G)$ a minimal continuous action and $m\in \Ntwo$. If  
    $(x_i,x_{i+1})\in Q_2(X,G)$ for each $i\in\{1,\ldots,m-1\}$ then $(x_1,\ldots,x_m)\in Q_m(X,G)$.
\end{theorem}

Let $(X,G)$ and $(Y,G)$ be two continuous actions. We say that $(Y,G)$ is a \textbf{factor} of $(X,G)$ through $\phi$, if $\phi:X\to Y$ is a surjective continuous map such that $\phi(gx)=g\phi(x)$, for all $x\in X$ and $g\in G$. In this case we also say $(X,G)$ is an \textbf{extension} of $(Y,G)$ through $\phi$, and that $\phi$ is a factor map. If $\phi$ is a homeomorphism,  we say that $(X,G)$ and $(Y,G)$ are \textbf{conjugate}, and $\phi$ is called a \textbf{conjugacy}.
A factor map $\phi:X\to Y$ is \textbf{proximal} if for all $(x,y)\in X\times X$ such that $\phi(x)=\phi(y)$ we have $(x,y)\in P(X,G)$. In this case we say $(X,G)$ is a \textbf{proximal extension} of $(Y,G)$.\\


We say that $(X,G)$ is equicontinuous if, for every $\varepsilon>0$, there exists $\delta >0$ such that for any $x,y\in X$, if $d(x,y)<\delta$ then $d(gx,gy)<\varepsilon$ for all $g\in G$. There exists an unique factor $(X_{eq},G)$ of $(X,G)$ such that $(X_{eq},G)$ is equicontinuous, and if $(Y,G)$ is an equicontinuous factor of $(X,G)$, then $(Y,G)$ is a factor of $(X_{eq},G)$. The factor $(X_{eq},G)$ is called the \textbf{maximal equicontinuous factor} and we denoted by $\pi_{eq}:X\to X_{eq}$ its factor map. \\

For a proof of the following result see \cite[Page 130]{auslander1988minimal}.
\begin{theorem}
[Ellis and Keynes]
\label{thm:RP=EQ}
    Let $G$ be an abelian group, $(X,G)$ a minimal continuous action, and $x,y\in X$. We have that $(x,y)\in Q_2(X,G)$ if and only if $\pi_{eq}(x)=\pi_{eq}(y)$.
\end{theorem}


We will now define several notions of rank. Coincidence rank was introduced in \cite{barge2006geometric}. Several papers have worked with bounds of the cardinality of the fibers of the maximal equicontinuous factor; we call them maximal and minimal rank following \cite{Aujogue2015}.
\begin{definition}
    Let $(X,G)$ be a continuous action and $y\in X_{eq}$. We define the \textbf{maximal rank ($r_M$), minimal rank ($r_m$)} and \textbf{coincidence rank ($r_c$)} as follows:
     \[  
     r_M(X,G)=\sup\{|\pi_{eq}^{-1}(y)|:y\in X_{eq} \},
     \] 
      \[
      r_m(X,G)=\inf\{|\pi_{eq}^{-1}(y)|:y\in X_{eq} \},
     \] 
    \[
       r_c(X,G)=\sup\{n\in\N :\exists x_1,\ldots,x_n\in \pi_{eq}^{-1}(y), (x_i,x_j)\notin P(X,G), \forall  i\neq j \}.
       \]
    
\end{definition}

If a continuous action is minimal, the coincidence rank does not depend on the choice of  $y$ (\cite[Lemma 2.10]{barge2013proximality}). This implies that $r_c(X,G)\leq r_m(X,G)$ as the coincidence
may be measured in a fiber of minimal size.
It is clear that $$r_c(X,G)\leq r_m(X,G) \leq r_M(X,G).$$

We have that $r_M(X,G)=1$ if and only if $(X,G)$ is equicontinuous ($\pi_{eq}$ is bijective). When $(X,G)$ is minimal and $r_m(X,G)=1$, we say that $(X,G)$ is \textbf{almost automorphic}. When $r_M(X,G)$ is finite, the maximal equicontinuous factor is finite to one. For minimal continuous actions, $r_c(X,G)=1$ if and only if $Q_2(X,G)=P(X,G)$ \cite{barge2013proximality}.

\section{Examples}
\label{sec:examples}
All the examples mentioned in this section are minimal. 

 A topological dynamical system is \textbf{uniquely ergodic} if there is only one invariant Borel probability measure. 

Subshifts with minimal rank 1 and arbitrary finite maximal rank were first constructed by Williams \cite{williams1984toeplitz}. These types of examples have been refined in several papers. Iwanik and Lacroix constructed uniquely ergodic examples \cite{iwanik1994some}; J{\"a}ger, Lenz, and Oertel constructed uniquely ergodic actions of $\R^d$ in the form of model sets \cite{jager2019model} (see also \cite{fuhrmann2021irregular}); and Bernales, Cortez, and Gomez constructed actions of residually finite groups \cite{bernales2024invariant} (see also \cite{gomez2025sequence}). All systems arising from substitutions of finite length have finite maximal rank. In particular, Shao defined primitive constant length substitutions (hence uniquely ergodic) that attain any possible maximal rank \cite{shao2005dynamical}.

Systems with bounded sequence entropy or bounded cardinality of IT tuples always have bounded minimal rank \cite{liu2025independence}. The classical Thue--Morse substitution system has minimal rank $2$ and maximal rank $4$. The proof follows from the fact that the Thue--Morse substitution system is an equicontinuous extension of a Toeplitz subshift (e.g. see \cite[Example 11.3]{hauser2024mean}. Generalized Morse (binary) sequences introduced by Keane \cite{keane1968generalized,lemanczyk1988toeplitz} are also two-to-one extensions of Toeplitz systems and hence of odometers. 

Keane also suggested that generalizations to other alphabets should exhibit higher-order symmetries \cite[Page 353]{keane1968generalized}. It is a folklore result that multivariate Thue--Morse sequences behave analogously to the classical one. For example, the system induced by the substitution $(0\to 012),(1 \to 120),(2 \to 201)$ admits a natural factor to a Toeplitz subshift, which is itself an extension of a ternary odometer. This factor is defined by summing three consecutive positions $\mod 3$ (rather than two). It is straightforward to verify that the factor is 3-to-1 and that it is an equicontinuous extension (as in \cite[Example 11.3]{hauser2024mean}). From this, we conclude that the system has minimal rank~$3$. For other examples with any possible finite minimal rank see \cite[Subsection 3.2]{ferenczi2016substitutions}.



In order to build other examples with finite coincidence rank we will use skew-product extensions. The next result follows from \cite[Theorem 3]{glasner1979construction} and by noting that those extensions are not finite-to-one at any point.

\begin{theorem}
    [Glasner and Weiss \cite{glasner1979construction}]
    \label{thm:GW}
    Let $(Y,S)$ be a minimal and uniquely ergodic topological dynamical system. There exists $(X,T)$, proximal extension of $(Y,S)$ through $\pi\colon X\to y$, such that $|\pi^{-1}(y)|=\infty$ for every $y\in Y$.
\end{theorem}

For similar smooth examples, see \cite{HAUPT_JAGER_2024}.

\begin{lemma}
\label{lem:ineq}
Let $(X,G)$ be a minimal proximal extension of $(Y,G)$. We have that $r_c(Y,G)\leq r_c(X,G)\leq r_m(Y,G)$.
\end{lemma}

\begin{proof}
First we will show that $(X_{eq},G)$ and $(Y_{eq},G)$ are conjugate. Let $\phi\colon X\to Y$ be a proximal factor map, $\pi_{eq}\colon X\to X_{eq}$ the maximal equicontinuous factor map of $(X,G)$ and $\pi'_{eq}\colon Y\to Y_{eq}$ the maximal equicontinuous factor map of $(Y,G)$. Since $(Y,G)$ is a factor of $(X,G)$, then $(Y_{eq},G)$ is a factor of $(X_{eq},G)$. 
Now note that there is a correspondence between $G$-invariant closed equivalence relations and factor maps (e.g. see \cite[Proposition 2.2]{garcia2025local}). Let $R_{\phi}\subset X^2$ be the closed $G$-invariant equivalence relation induced by $\phi$. Since $\phi$ is proximal, then $$R_{\phi}\subset P(X,G)\subset Q_2(X,G).$$ We have that $(X_{eq},G)$ is a factor of $(Y,G)$ (the factor map is given by the closed $G$-invariant equivalence relation generated by $Q_2(X,G)$ on $Y=X/R_{\phi}$). Thus $(X_{eq},G)$ is a factor of $(Y_{eq},G)$. Since the maximal equicontinuous factor is unique, we have that $(X_{eq},G)$ and $(Y_{eq},G)$ are conjugate. 

First we prove that $r_c(Y,G)\leq r_c(X,G)$. Let $k=r_c(X,G)$, $y_0\in Y_{eq}$ and $y_1,\ldots,y_{k+1}\in {\pi'^{-1}_{eq}}(y_0)$. Since $\phi$ is surjective, for all $y_i$ there exists $x_i\in X$ such that $\phi(x_i)=y_i \mbox{ for all } i \in \{1,\ldots,k+1\}$. Up to homeomorphism, we have $\pi'_{eq}\circ \phi=\pi_{eq}$. This implies that $\pi_{eq}(x_i)=y_0$, for all $i\in\{1,\ldots,k+1\}$. Since $r_c(X,G)=k$, there exists $x_i\neq x_j$ such that $(x_i,x_j)\in P(X,G)$. Using the continuity of $\phi$, we conclude that $(\phi(x_i),\phi(x_j))\in P(Y,G)$. This implies that $r_c(Y,G)\leq r_c(X,G)$.

Now we prove the other inequality. Let $x\in X_{eq}$ and $k=r_m(Y,G)$. Then there exists $y_0\in Y_{eq}$ such that $k=|\pi_{eq}'^{-1}(y_0)|$. Let $x_1,\dots,x_{k+1}\in {\pi}_{eq}^{-1}(x).$ Since $\pi'_{eq}\circ \phi={\pi}_{eq}$ (up to homeomorphism), there exist $i,j\in \{1,\dots,k+1\}$, with $i\neq j$, such that $\phi(x_i)=\phi(x_j)$. Since $\phi$ is proximal, we have that $(x_i,x_j)\in P(X,G)$. We conclude that $r_c(X,G)\leq r_m(Y,G)$.
\end{proof}

The examples mentioned in the first paragraphs of this section appear to have minimal rank equal to coincidence rank \cite[Proposition 4.9]{Aujogue2015}. In fact, the existence of minimal actions where the coincidence rank and minimal rank differ was explicitly posed as a question in \cite[Page 38]{Aujogue2015}. Using those examples together with the following result, one can construct many such cases.
\begin{corollary}
\label{cor:ex}
    Let $(Y,S)$ be a minimal uniquely ergodic topological dynamical system. There exists a minimal topological dynamical system $(X,T)$ with $$r_c(Y,S)\leq r_c(X,T)\leq r_m(Y,S) \mbox{, and }r_m(X,T)=\infty.$$
\end{corollary}

\begin{proof}
 By Theorem \ref{thm:GW}, there exists $(X,T)$, a proximal extension of $(Y,S)$ through $\pi\colon X\to Y$, such that $|\pi^{-1}(y)|=\infty$ for every $y\in Y$. This implies that $r_m(X,T)=\infty$. Using Lemma \ref{lem:ineq} we conclude that $r_c(X,T)=r_m(Y,S)$.
\end{proof}

\section{Finite maximal rank}

\subsection{$m$-Equicontinuity}





\begin{definition}
 Let $(X,G)$ be a continuous action, $x\in X$ and $m\in \Ntwo$. We say that $x$ is an \textbf{m-equicontinuity point} if for every $\varepsilon>0$ there exists $\delta>0$ such that for any $x_1,\dots,x_{m}\in B_{\delta}(x)$ we have that for every $g\in G$ there exist $i,j\in\{1,\ldots,{m}\}$, with $i\neq j,$ such that $d(gx_i,gx_j)< \varepsilon$. We denote the set of $m$-equicontinuity points by $E^m(X,G)$.

We say that $(X,G)$ is $m$\textbf{-equicontinuous} if $E^m(X,G)=X.$ 
The action $(X,G)$ is \textbf{almost $m$-equicontinuous} if $E^m(X,G)$ is a residual subset (countable intersection of open dense subsets) of $X$.
\end{definition}
Note that $(X,G)$ is equicontinuous if and only if it is $2$-equicontinuous. This $2$ represents the fact that equicontinuity is defined on pairs. 
\begin{proposition}
\label{openE}
Let $(X,G)$ be a continuous action and $m\in \Ntwo$.
We have that $gE^m(X,G) = E^m(X,G)$ for all $g\in G$.
\end{proposition}

\begin{proof}
Let $g_0\in G, x\in g_0E^m(X,G)$ and $\varepsilon>0$. Let $y\in E^m(X,G)$, with $x=g_0y$. There exists $0<\delta<\varepsilon$ such that if $x_1,\ldots,x_{m}\in B_\delta(y)$, then for any $g\in G$, there exists $i,j\in \{1,\dots,m\}$, with $i\neq j$, such that $d(gx_i,gx_j)<\varepsilon$. On the other hand, using the continuity of the action, there exists $\eta>0$ such that for all $z\in B_{\eta}(x)$, we have that $d(g_0^{-1}x,g_0^{-1}z)<\delta$.
Let $x_1,\ldots,x_{m}\in B_{\eta}(x)$ and $h\in G$. For every $i\in\{1,\ldots,m\}$, we have that $g_0^{-1} x_i\in B_\delta(y)$.
Hence, for every $h\in G$, there exist $i,j\in \{1,\ldots,m\}$, with $i\neq j$, such that 
\[
d(hg_0g_0^{-1}x_i,hg_0g_0^{-1}x_j)=d(hx_i,hx_j)<\varepsilon.
\]
This implies that $x\in E^m(X,G).$ Thus $g_0E^m(X,G) \subset E^m(X,G)$. Similarly, we have that $g_0^{-1}E^m(X,G) \subset E^m(X,G)$. Hence $gE^m(X,G) = E^m(X,G)$.
\end{proof}

Given a continuous action $(X,G)$, $m\in \Ntwo$ and $\varepsilon>0$, we define 
\begin{align}
\label{eq1}
\begin{split}
E^m_\varepsilon(X,G) =  
\{x\in X: \exists \delta \text{ s.t. } \forall x_1,\ldots,x_{m}\in B_\delta(x), \\
\forall g\in G, \exists 1\leq i\neq j\leq m, \text{ s.t. } d(gx_i,gx_j) < \varepsilon\}.
\end{split}
\end{align}

\begin{lemma}
\label{inverselyinv}
Let $(X,G)$ be a continuous action, $\varepsilon>0$ and $m\in \Ntwo$. We have that
$E^m_\varepsilon(X,G)$ is open,  $gE_{\varepsilon}^m(X,G) = E_{\varepsilon}^m(X,G)$ for all $g\in G$, and
$$E^m(X,G)=\bigcap_{n\in\N} E^m_{1/n}(X,G).$$
\end{lemma}

\begin{proof}

One can show that $gE_{\varepsilon}^m(X,G) = E_{\varepsilon}^m(X,G)$ for all $g\in G$ with a similar argument as in the proof of Proposition \ref{openE}. 
Now, we will show that $E^m_\varepsilon(X,G)$ is open. Let $x\in E_\varepsilon^m(X,G)$, there exists $\delta>0$, such that $\delta$ satisfies the condition \eqref{eq1} for $x$. Then $B_{\frac{\delta}{2}}(x)\subset E_\varepsilon^m(X,G)$. 
    Indeed if $y\in B_{\frac{\delta}{2}}(x)$, $x_1,\ldots,x_{m}\in B_{\frac{\delta}{2}}(y)$ 
    and $g\in G$ there exists $i,j\in\{1,\ldots,m\}$, with $i\neq j$, such that $ d(gx_i,gx_j)< \varepsilon$. This shows that $E^m_\varepsilon(X,G)$ is open. 
    Finally, we prove that $E^m(X,G)=\bigcap_{n\in \N} E^m_{1/n}(X,G)$. Clearly \[\bigcap_{n\in \N} E^m_{1/n}(X,G) \subset E^m(X,G).\]
    Conversely, let $x\in E^m(X,G)$ and $n\in \N$. There exists $\delta>0$ such that for any $x_1,\ldots,x_{m}\in B_\delta(x)$, and any $g\in G$, there exists $i,j\in\{1,\ldots,m\}$, with $i\neq j$, such that $d(gx_i,gx_j)<\frac{1}{n} $, thus $x\in E^m_{1/n}(X,G).$ 
\end{proof}

The next lemma follows from the previous result. 
\begin{lemma}
    Let $(X,G)$ be continuous action and $m\in\Ntwo$. We have that $E^m(X,G)$ is dense if and only if $(X,G)$ is almost $m$-equicontinuous. 
\end{lemma}

\begin{proposition}
\label{almostequicontinuo}
    Let $(X,G)$ be a minimal continuous action and $m\in \Ntwo$. If $(X,G)$ is almost $m$-equicontinuous then $(X,G)$ is $m$-equicontinuous. 
\end{proposition}
\begin{proof}
 By Lemma \ref{inverselyinv}, we have that $E^m(X,G)=\bigcap_{n\in \N} E^m_{1/n}(X,G)$ and that $E^m_{1/n}(X,G)$ is a $G$-invariant open non-empty (by hypothesis) subset for every $n\in \N$. Thus $X\setminus E^m_{1/n}(X,G)$ is a closed $G$-invariant subset. Since $(X,G)$ is minimal, then $X\setminus E^m_{1/n}=\emptyset$. We conclude that $$E^m(X,G)=\bigcap_{m\in\N} E^m_{1/n}(X,G)=X.$$
\end{proof}

For $m\in \N$, we define 
\[
 \Delta^{({m})}(X)=\{(x_1,\dots,x_m)\in X^m:x_i=x_j \mbox{ for some } i\neq j\in \{1,\dots,m\}\}.
\]
\begin{theorem}
    \label{requicontinuous1}
    Let $(X,G)$ be a minimal continuous action and $m\in \Ntwo$. Then $(X,G)$ is $m$-equicontinuous if and only if $Q_{m}(X,G)\setminus \Delta^{({m})}(X)=\emptyset$.
\end{theorem}
\begin{proof}
 
First assume that $Q_{m}(X,G)\setminus \Delta^{(m)}(X)\neq\emptyset$. Let \[
(x_1,\dots,x_{m})\in Q_{m}\setminus \Delta^{(m)}(X).
\]
For all $i\in \{1,\ldots,m\}$, there exists a closed neighborhood $A_i$ of $x_i$, such that the collection of sets $\{A_1,\dots, A_m\}$ is disjoint. 
Let \[\delta=\min\{d(A_i,A_j):i\neq j \in\{1,\ldots,m\}\}.\] 
Let $U$ be an open subset of $X$. Using the minimality of $(X,G)$ and the compactness of $X$ we conclude there is a finite $F\subset G$ such that $X=\bigcup_{g\in F}gU.$
Let $\delta'>0$ be a Lebesgue number for the open cover $\{gU:g\in F\}$ (every subset of $X$ with diameter less than $\delta'$ is contained in some member of the cover). 
By hypothesis, for any $0 < \eta \leq \delta$ there are $x_i'\in A_i$ and $g_0\in G$ with $$\max_{1\leq i<j\leq m}d(g_0 x'_i,g_0x'_j)<\delta'.$$
Thus there exists $h\in F$ with $\{g_0x'_1,\ldots,g_0x'_{m}\}\subset hU$. 
Let $y_i=h^{-1}g_0x_i'\in U$ for each $i\in\{1,\ldots,m\}.$ 
Note that \[d(g^{-1}_0hy_i,g^{-1}_0hy_j)\geq d(A_i,A_j)\geq \delta\] when $i\neq j$. Thus, $(X,G)$ is not $m$-equicontinuous.

Now suppose that $(X,G)$ is not $m$-equicontinuous.
 There exist $x_0\in X$ and $\varepsilon_0>0$ such that for 
every $n\in \N$, there exist $x^n_1,\ldots,x^n_{m}\in B_{1/n}(x_0)$ and $g_n\in G$ such that $d(g_nx_i^n ,g_nx_j^n)>\varepsilon$, for every $i,j\in \{1,\dots,m\}$, with $i\neq j$. By compactness of $X$, there exists an increasing sequence of natural numbers, $(k_n)_{n\in\N}$ and $y_1,\dots,y_m\in X$ such that 
\[
\lim_{n\to \infty}g_{k_n} x_i^{k_n}= y_i \mbox{ for all } i\in \{1,\ldots,m\}.
\] 

Furthermore, one can check that 
$d(y_i,y_j)\geq \varepsilon$ for all $i,j\in\{1,\ldots,m\}$, with $i\neq j$. 
Let $\delta>0$. There exists $N\in\N$ such that $1/N\leq \delta$ and 
\[
d(y_i,g_{k_N}x_i^{k_N})<\delta\mbox{ for all }i\in \{1,\ldots,m\}.
\]
For every $i\in \{1,\ldots,m\}$, we define $y'_i:=g_{k_N}x_i^{k_N} $. 
We have that 
\[
d(g_{k_N}^{-1}y'_i,g_{k_N}^{-1}y'_j)=d(g_{k_N}^{-1}g_{k_N}x_i^{k_N},g_{k_N}^{-1}g_{k_N}x_j^{k_N})=d(x^{k_N}_i,x^{k_N}_j)<\delta
\]
 for all $i,j\in\{1,\ldots,m\}, i\neq j$. Furthermore, $(y_1,\ldots,y_m)\in Q_{m}(X,G)\setminus \Delta^{(m)}(X)$.
 
\end{proof}



\begin{theorem}
   Let $G$ be an abelian group, $(X,G)$ a minimal continuous action and $m\in \Ntwo$. The action $(X,G)$ is $m$-equicontinuous if and only if $r_M(X,G)\leq m-1$.
 \end{theorem}

\begin{proof}
     First assume that $r_M(X,G)\leq m-1$.
     By Theorem \ref{requicontinuous1}, it is sufficient to prove that $Q_{m}(X,G)\setminus \Delta^{(m)}(X)=\emptyset$. Suppose that there exists $(x_1,\ldots,x_{m})\in Q_{m}(X,G)\setminus \Delta^{(m)}(X)$. This implies that $x_i\neq x_j$ and $\pi_{eq}(x_i)=\pi_{eq}(x_j)$ for all $i,j\in \{1,\ldots,m\}$, $i\neq j$ (use Theorem \ref{thm:RP=EQ}). In the other hand, by hypothesis, there exist $i,j\in\{1,\ldots,m\}, i\neq j$ such that  $x_i=x_j$, which is a contradiction. Hence, $(X,G)$ is $m$-equicontinuous.
    
     
     Now suppose that $(X,G)$ is $m$-equicontinuous. Let $y\in X_{eq}$. We will prove that $|\pi_{eq}^{-1}(y)|\leq m-1$. Assume that $|\pi_{eq}^{-1}(y)|\geq m$. Then, there exist $x_1,\ldots,x_{m}\in X$ such that $x_i\neq x_j$ and  $\pi_{eq}(x_i)=\pi_{eq}(x_j)$ for all $i,j\in \{1,\ldots,m\}$. This implies that $(x_i,x_j)\in Q_2(X,G)$, for all $i,j\in \{1,\ldots,m\}$. Use Theorem \ref{nregionally} to conclude that $(x_1,\ldots,x_{m})\in Q_{m}(X,G)\setminus \Delta^{(m)}(X)$. Thus $(X,G)$ is not $m$-equicontinuous.
\end{proof}

\subsection{$m$-Sensitivity}
\begin{definition}
Let $(X,G)$ be a continuous action and $m\in\Ntwo$. We say $(X,G)$ is \textbf{$m$-sensitive} if there exists $\varepsilon>0$ such that for any non-empty open set $U$, there are distinct points $x_1,\ldots,x_{m}\in U$ and $g\in G$ with 
$$\min\{d(gx_i,gx_j):i\neq j\in\{1,\ldots,m\}\}\geq \varepsilon.$$ 
Such an $\varepsilon>0$ is called an $m$-sensitivity constant of $(X,G)$.
\end{definition}

We say that an action $(X,G)$ is \textbf{transitive} if for any two non-empty open sets $U$ and $V$ there exists $g\in G$ such that $U\cap gV\neq \emptyset.$

\begin{theorem}
\label{dicotomia0}
Let $m\in \Ntwo$. A transitive continuous action is either $m$-sensitive or almost $m$-equicontinuous.
\end{theorem}

\begin{proof}
Assume that $(X,G)$ is not almost $m$-equicontinuous, thus $E^m(X,G)$ is not residual. Since $E^m(X,G)=\bigcap_{n\in \N} E^m_{1/n}(X,G)$ and $ E^m_{1/n}(X,G)$ is open for every $n\in \N$ (Lemma \ref{inverselyinv}), there exists $N\in \N$ such that $E^m_{1/N}(X,G)$ is not dense. Assume that $E^m_{1/N}(X,G)$ is non-empty. Then $U=X\setminus \overline{E^m_{1/N}(X,G)}$ is open and non-empty. By Lemma \ref{inverselyinv}, $gE^m_{1/N}(X,G)= E^m_{1/N}(X,G)$ for every $g\in G$. By transitivity, there exist $g \in G$ such that, 
$$\emptyset\neq U\cap gE^m_{1/N}(X,G)= U\cap E^m_{1/N}(X,G)=\emptyset,$$ 
a contradiction. Thus $E^m_{1/N}(X,G)=\emptyset$. This implies that for all $\delta>0$ and for every $x\in X$, there exists $x_1,\ldots,x_{m}\in B_\delta(x)$ and $g\in G$ such that for all $i,j\in \{1,\ldots,m\}$, with $i\neq j$, we have that $d(gx_i,gx_j)\geq 1/N$. Thus the system is $m$-sensitive with $m$-sensitivity constant $1/N$.

Now suppose that $(X,G)$ is almost $m$-equicontinuous. By Lemma \ref{inverselyinv}, $E^m(X,G)=\bigcap_{n\in \N} E^m_{\frac{1}{n}}(X,G)$. Given $\varepsilon>0$, there exists $N\in \N$ such that $1/N<\varepsilon$. Let $x\in E^m_{\frac{1}{N}}(X,G)$ this implies that there exists $\delta>0$ such that for every $x_1,\ldots,x_{m}\in B_\delta(x)$ and $g\in G$, we have that
\[\min\{d(gx_i,gx_j):i\neq j\in\{1,\ldots,m\}\}< \varepsilon.\] Thus $(X,G)$ is not $m$-sensitive.

\end{proof}

Using Proposition \ref{almostequicontinuo} and Theorem \ref{dicotomia0} we obtain the following.
\begin{corollary}
\label{requicontinuous}
Let $m\in \Ntwo$. A minimal continuous action is $m$-equicontinuous if and only if it is not $m$-sensitive. 
\end{corollary}

The previous result is the multivariate version of the Auslander-Yorke dichotomy \cite{auslander1980interval}. For another multivariate Auslander-Yorke type dichotomy see \cite[Corollary 4.6]{breitenbucher2024multivariate}.





We recover the following result from \cite[Theorem 3.6]{Shao2008}. 
\begin{theorem}
Let $(X,G)$ be a minimal continuous action and $m\in\Ntwo$. Then $(X,G)$ is $m$-sensitive if and only if $r_M(X,G)\geq m$.
\end{theorem}
 
	

 \subsection{Maximal $m$-equicontinuous factors}

In this subsection we show that each action has a family of, possibly different, maximal $m$-equicontinuous factors, which we will call $m$-MEF. 

Given a function $f\colon X\to Y$ and $n\in\N$ we denote the product function by $f^{(n)}\colon X^n\to Y^{n}$. Using \cite[Proposition 15.6]{ellis2014automorphisms} and Theorem \ref{nregionally} we conclude the following result. 

\begin{proposition}\label{regproxim} 
Let $(X,G)$ be a minimal continuous action, $(X',G)$ a factor through $\pi\colon X\to X'$ and $m\in \Ntwo$. Then $\pi^{(m)}(Q_m(X,G))=Q_m(X',G).$ \end{proposition}

Given an equivalence relation $R\subset X^2$, we denote the associated quotient space by $X/R$.  
Let $\varphi\colon X\to Y$ be a surjective function. We denote the equivalence relation, $x\thicksim y \iff \varphi(x)=\varphi(y)$, by $R_{\varphi}.$ For $x\in X$, $[x]_{\varphi}\in X/R$ is the class of the points related to $x$.
\begin{definition}
  Let $m\in \Ntwo$, $(X,G)$, $(X',G)$ be two minimal continuous action and $\pi:X\to X'$. We say that $\pi$ is an $m$\textbf{-MEF} map if for all $A\in X/ \pi_{eq}$ there exists $\{B_i\}_{i=1}^{m-1}$,  $B_i\in X/ \pi$ such that $A=\bigcup_{i=1}^{m-1}B_i.$
 \end{definition}

 \begin{lemma}
Let $m\in \Ntwo$, $G$ be abelian, $(X,G)$,$(X',G)$ continuous actions and $\pi:X\to X'$ a factor. If $\pi$ is an $m$-MEF map, then $(X',G)$ is $m$-equicontinuous.
 \end{lemma}

 \begin{proof}
     By Theorem \ref{requicontinuous1} it is sufficient to prove that $Q_m(X',G)\subset \Delta^{(m)}(X')$. Let $(x'_1,\ldots,x'_m)\in Q_m(X',G)$. Using Proposition \ref{regproxim} there exists $(x_1,\ldots,x_m)\in Q_m(X,G)$ such that $\pi^{(m)}(x_1,\ldots,x_m)=(x'_1,\ldots,x'_m).$
     Let $A=[x_1]_{R_{\pi_{eq}}}\in X/R_{\pi_{eq}}$. By Theorem \ref{thm:RP=EQ}, we have that $x_1,\dots, x_m\in A$.
    
     Since $\pi$ is an $m$-MEF map, there exist $B_1,\dots,B_{m-1}\in X/\pi$ with $A=\bigcup_{i=1}^{m-1}B_i$. Hence, there exist $i\neq j\in\{1,\ldots,m\}$ and $k\in \{1,\ldots,m-1\}$ such that $x_i,x_j\in B_k$. This implies that $\pi(x_i)=\pi(x_j)$, and hence $x'_i=x'_j$. We conclude that $(x'_1,\ldots,x'_m)\in \Delta^{(m)}(X')$. Furthermore, $(X',G)$ is $m$-equicontinuous.
 \end{proof}

 \begin{proposition}
Let $m\in \Ntwo$, and $(Y,G)$ a factor of $(X,G)$ through $\phi:X\to Y$.
If $(Y,G)$ is $m$-equicontinuous, then there exists an $m$-MEF map, $\pi:X\to X'$, so that $(Y,G)$ is also a factor of $(X',G)$.
 \end{proposition}

 \begin{proof}
Let $X'=X/(R_{\phi}\cap R_{\pi_{eq}})$ and $\pi:X\to X'$ the continuous projection. There exists an action $(X',G)$ that makes $\pi$ a factor map. Clearly $(Y,G)$ is also a factor of $(X',G)$. By Proposition \ref{requicontinuous1}, we have that $Q_m(Y,G)\subset \Delta^{(m)}(Y)$. This implies that $\pi$ is an $m$-MEF map. 
 \end{proof}

 \section{Finite coincidence rank}







Let $K\subset G$ be a compact subset. We say that $A\subset G$ is \textbf{$K$-convering} if $G=KA$. Note that $A\subset G$ is {syndetic} if and only if there exists a compact subset $K'\subset G$ such that $A$ is $K'$-covering.\\

\begin{definition}
    Let $(X,G)$ be a continuous action and $m\in \Ntwo$. We say $x\in X$ is a \textbf{convering $m$-equicontinuity point} if for every $\varepsilon>0$, there exists $U$, an open neighborhood of $x$, and $K\subset G$ compact such that for every $x_1,\dots,x_m\in U$, we have that
    \[
    \{g\in G: \exists i\neq j \mbox{ s.t. } d(gx_i,gx_j)\leq\varepsilon\}
    \]
    is $K$-convering. 

    We say $(X,G)$ is \textbf{cover $m$-equicontinuous} if every $x\in X$ is a convering $m$-equicontinuity point. 
\end{definition}

Note that every $m$-equicontinuity point is a covering $m$-equicontinuity point. 

The previous definition is related but does not coincide with the definition of syndetic equicontinuity from \cite{huang2018analogues}. That definition is actually related to minimal rank. 


\begin{definition}
Let $(X,G)$ be a continuous action and $m\in\Ntwo$. We say $(X,G)$ is \textbf{compactly $m$-sensitive} if there exists $\varepsilon>0$ such that for each non-empty open set $U\subset X$ and $K\subset G$ compact there exist $x_1,\ldots,x_m \in U$ and $h\in G$ such that \[Kh \subset \{g\in G: d(gx_i, g x_j)>\varepsilon, \forall i\neq j \;\;  i,j\in \{1,\ldots,m\} \}.\]
\end{definition}

When $G=\Z$ the previous definition coincides with the notion of block $m$-sensitivity as defined in \cite{zou2017stronger}. Originally defined for $m=1$ in \cite{ye2018sensitivity}.

\begin{proposition}
    Let $(X,G)$ be a minimal continuous action and $m\in \Ntwo$. Then $(X,G)$ is cover $m$-equicontinuous if and only if it is not compactly $m$-sensitive. 
\end{proposition}
\begin{proof}
    Assume that $(X,G)$ is not compactly $m$-sensitive. For every $\varepsilon>0$, there exists a non-empty open subset $V\subset X$ and a compact subset $K\subset G$ such that for every $y_1,\ldots,y_m \in V$, we have that 
    \[Kh \nsubseteq \{g\in G: d(gy_i, g y_j)>\varepsilon, \forall i\neq j, \;\;  i,j\in \{1,\ldots,m\} \}\]
    for every $h\in G$. This implies that
     \[
    \{g\in G: \exists i\neq j \mbox{ s.t. } d(gy_i,gy_j)\leq\varepsilon\}
    \]
    is $K$-convering. 
    Let $x\in X$. Since $(X,G)$ is minimal we have that $Gx$ is dense. There exists $U$, an open neighborhood of $x$ and $g_x\in G$, such that $g_xU\subset V$. Since $K$-convering sets are invariant under translation we conclude that 
    for every $x_1,\dots,x_m\in U$, we have that
    \[
    \{g\in G: \exists i\neq j \mbox{ s.t. } d(gx_i,gx_j)\leq\varepsilon\}
    \]
    is $K$-convering.

    Now assume that $(X,G)$ is cover $m$-equicontinuous. For every $\varepsilon>0$, there exists $U$, an open neighborhood of $x$, and a compact subset $K\subset G$ such that for every $x_1,\dots,x_m\in U$, we have that
    \[
    \{g\in G: \exists i\neq j \mbox{ s.t. } d(gx_i,gx_j)\leq\varepsilon\}
    \]
    is $K$-convering. 
    This implies that for every $h\in G$
    \[Kh \nsubseteq \{g\in G: d(gy_i, g y_j)>\varepsilon, \forall i\neq j \;\;  i,j\in \{1,\ldots,m\} \}.\]
We conclude that $(X,G)$ is not compactly $m$-sensitive.    
\end{proof}


Let $(X,G)$ be a continuous action and $U,V$ open non-empty subsets of $X$. We define 
\[
N(U,V)=\{g\in G: U\cap gV \neq \emptyset\}.
\]
The next results follows from the proof of \cite[Proposition 3.7]{maass_2007}. We write the proof for completeness. 
\begin{lemma}
\label{lem:weakmixing}
    Let $(X,G)$ be a minimal continuous action and $m\in\Ntwo$. If \[(x_1,\dots, x_m)\in Q_m(X,G)\] and $U_1\times \cdots \times U_m$ is a product neighborhood of $(x_1,\dots, x_m)$ then
    \[
    \bigcap_{i=1}^{m}N(U_1,U_i)\neq \emptyset.
    \]
\end{lemma}
\begin{proof}
We define $L(x_1,\dots, x_m)$ as the set of points $x\in X$ such that for each product neighborhood  $V\times V_1\times \cdots \times V_m$, of $(x,x_1,\dots, x_m)$, there exists $y_1,\dots,y_m\in V$ and $g\in G$ such that $gy_i\in V_j$ for all $j\in\{1,\dots,m\}$. 

One can check that $L(x_1,\dots, x_m)$ is non-empty, closed and $G$-invariant. Since $(X,G)$ is minimal we have that $L(x_1,\dots, x_m)=X$. Thus, $x_1\in L(x_1,\dots, x_m)$. The conclusion follows from this fact. 

\end{proof}

The equivalence between (1) and (3) in Theorem \ref{blockrproximal} can possibly be obtained by adapting the proof of \cite[Theorem 1.1]{zou2017stronger}, which is the multivariate version of \cite[Theorem A]{ye2018sensitivity}. Here, we present an alternative proof that employs different tools, including (2) of Theorem \ref{blockrproximal} and Lemma \ref{lem:weakmixing}. 
 
\begin{theorem}
\label{blockrproximal}
Let $G$ be a $\sigma$-compact, abelian group, $(X,G)$ a minimal continuous action and $m\in\Ntwo$. Then the following conditions are equivalent.
\begin{enumerate}
    \item $(X,G)$ is compactly $m$-sensitive.
    \item There exists $\delta>0$ such that for each open set $U\subset X$, there exist $x_1,\ldots,x_m\in X$ such that $x_1\in U$, $(x_i,x_j)\in Q_2(X,G)$ and 
$$\inf_{g\in G} d(gx_i,gx_j)>\delta, \mbox{ for all } i,j\in \{1,\ldots,m\}.$$
\item $r_c(X,G)\geq m$.
\end{enumerate}
\end{theorem}
\begin{proof}

$(3) \implies (2)$

By hypothesis there exists $(z_1,\dots,z_m)\in Q_m(X,G)$ such that $(z_i,z_j)\notin P(X,G)$ for every $i\neq j\in \{1,\dots,m\}$.
For every $i,j\in \{1,\ldots,m\}$, define $\delta_{ij}=\inf_{g\in G}d(gz_i,gz_j)$. Let $\delta=\frac{1}{2}\min\{\delta_{ij}: i,j\in \{1,\ldots,m\}\} >0$, $U$ a non-empty open subset of $X$ and $x_1\in U$. Since $Gz_1$ is dense, there exists $\{g_n\}_{n\in\N}$ such that $\lim_{n\to\infty}g_{n}z_1=x_1$. Since $X$ is compact assume, without loss of generality, that there exists $x_2,\dots,x_m\in X$ such that $\lim_{n\to\infty}g_{n}z_i=x_{i}, \forall i\in\{2,\ldots,m\}$.

Since $(z_i,z_j)\in Q_2(X,G)$ and $Q_2(X,G)$ is closed and $G$-invariant, then
\[(x_i,x_j)\in Q_2(X,G)\] 
for all $i,j\in \{1,\ldots,m\}$. For other hand, since the action is continuous we conclude that for all $ i,j\in\{1,\ldots,m\}, i\neq j$

\begin{align*}
    \inf_{g\in G} d(gx_i,gx_j)&=\lim_{n\to \infty}d(gg_nz_i,gg_nz_j)\\
    &\geq\inf_{g\in G}d(gz_i,gz_j)>\delta.
\end{align*}

$(2) \implies (1)$  

    Let $\delta$ be a constant obtained from condition $(2)$, $U\subset X$ open, and $K\subset G$ compact. There exist $x_1,\ldots,x_m\in X$, with $x_1\in U$ such that $(x_i,x_j)\in Q_2(X,G)$  and 
\begin{equation}
    \inf_{g\in G}d(gx_i,gx_j)>\delta\mbox{ } \forall \; i,j\in\{1,\ldots,m\}. \tag{A} \label{eq:A}
\end{equation}
   
    Choose $U_1\times \cdots \times U_m$ a sufficiently small open product neighborhood of $(x_1,\dots,x_m)$ so that $U_1\subset U$ and
    \[ 
    \min_{g\in K}d(gU_i,gU_j)>\delta \mbox{ }\forall \; i,j\in\{1,\ldots,m\}.
    \]
    Since $(x_i,x_j)\in Q_2(X,G)$ for all $i,j\in \{1,\ldots,m\},$ using Theorem \ref{nregionally} we obtain that \[(x_1,\ldots,x_m)\in Q_m(X,G).\] From Lemma \ref{lem:weakmixing} we get that
    \[\bigcap_{i=1}^{m}N(U_1,U_i)\neq\emptyset.\] 
    Let $h\in \bigcap_{i=1}^{m}N(U_{1},U_i).$ For every $i\in \{1,\ldots,m\}$, there exists $y_i\in U_{1}$ such that $hy_i\in U_i$. Using \eqref{eq:A} we have that
    \[\
    \min _{g\in K}d(ghy_i,ghy_j)\geq \min_{g\in K}d(g U_i,gU_j)>\delta\] for every $i,j\in \{1,\ldots,m\}$. Hence \[Kh\subset \{g\in G:d(gy_i,gy_j)>\delta\}.\] We conclude that $(X,G)$ is compactly $m$-sensitive.

    (1) $\implies$ (3)
    
    Let  $d_{eq}$ be a compatible metric of $X_{eq}$. For each $k\in\N$, we set $\varepsilon_k>0$ so that $\lim_{k\to \infty}\varepsilon_k= 0$. Since $(X_{eq},G)$ is equicontinuous, for each $k\in\N$, there exist $0<\tau_k$, $\tau'_k<\varepsilon_k$ such that for every $y_1,y_2\in X_{eq}$, with $d_{eq}(y_1,y_2)<\tau_k$, then $d_{eq}(gy_1,gy_2)<\varepsilon_k$ for every $g\in G$; and for every $x_1,x_2\in X$ with $d(x_1,x_2)<\tau_k'$ then $d_{eq}(\pi (x_1),\pi (x_2))<\tau_k$.  

Let $x\in X$ and $U_k=B_{\tau'_k}(x)$. We consider a sequence of compact sets, $\{F_n\}_{n\in \N}$, such that $F_i\subset F_{i+1}$ for every $i\in \N$, and $\cup_{n\in\N}F_n=G$. For all $g\in G$, there exists $n_g\in\N$ such that $g\in F_n$ for all $ n\geq n_g$. 

Since $(X,G)$ is compactly $m$-sensitive, there exists $\delta>0$ so that for each $n,k\in\N$, there exists $h_{n,k}\in G$ and $\{x^k_1,\ldots,x^k_m\}\subset U_k$ such that 

\begin{equation}
 F_nh_{n,k}\subset \{g\in G:d(gx_i^k,gx_j^k)>\delta,\forall i\neq j\}.
 \tag{B} \label{eq:B}
\end{equation}

For every $k\in\N$ and $i\in\{1,\dots,m\}$, there exists $z_i^k\in X$ so that $$\lim_{n\to\infty} h_{n,k}x_i^k= z_i^k.$$ 

Note that for every $g\in G$, if $n\geq n_g$ then $gh_{n,k}\in F_nh_{n,k}$. Thus, for every $g\in G$, $n\geq n_g$, $k\in\N$ we have $d(gh_{n,k}x_i^k, gh_{n,k}x_j^k)>\delta$ (using $\ref{eq:B}$) and hence $d(gz_i^k,gz_j^k)\geq \delta$ for all $i\neq j\in\{1,\ldots,m\}$ .

 
 
For every $i\in \{1,\ldots,m\}$, let $z_i\in X$ be an accumulation point of $\{z_i^k\}_{k\in \N}$.
We have that $d(gz_i,gz_j)\geq \delta$ for all $i,j\in \{1,\ldots,m\}$ with $i\neq j$; thus $(z_i,z_j)$ are not proximal.  
Since $x^{k}_j\in U_k$ for $j\in\{1,\ldots,m\}$, we have $d_{eq}(\pi (x^{k}_i),\pi (x^{k}_j) )<\tau_k$ and $d_{eq}(g\pi (x^{k}_i),g\pi (x^{k}_j))<\varepsilon_k$ for each $g\in G$.
In particular, 
$$d_{eq}(h_{n,k}\pi (x^{k}_i),h_{n,k}\pi (x^{k}_j))<\varepsilon_k,$$
for each $n\in \N$  and $i,j\in\{1,\ldots,m\}$. 
This implies that $d_{eq}(\pi (z^{k}_i),\pi (z^{k}_j))\leq \varepsilon_k$ for all $i,j\in \{1,\ldots,m\}$, with $i\neq j$. Hence $d_{eq}(\pi(z_i),\pi(z_j))<\varepsilon_k$ for all $i,j\in\{1,\ldots,m\}$ and $k\in\N$.
We conclude that for every $i,j\in\{1,\ldots,m\}$, we have $\pi(z_i)=\pi(z_j)$. This implies that $r_c(X,G)\geq m$.
\end{proof}




\begin{corollary}
      Let $m\in\Ntwo$, $G$ be abelian and $\sigma$-compact, and $(X,G)$ a minimal continuous action. Then $(X,G)$ is cover $m$-equicontinuous if and only if $r_c(X,G)<m$.
\end{corollary}

\section{Questions}

As we noted, the Thue--Morse substitution induces a system with maximal rank 4 and minimal rank 2. 

\begin{question}
    Are there minimal continuous actions $(X,G)$ where $r_m(X,G)=r_M(X,G) <\infty $?
\end{question}

The examples given by Corollary \ref{cor:ex} work for actions of $\Z$. Examples like this should be possible for other groups, but the techniques to build them would be different. 

\begin{question}
For which groups can we find examples with  $r_c(X,G)<r_m(X,G)$?
\end{question}

More generally:
\begin{question}
    Which are the possible combinations of minimal rank, maximal rank, and coincidence rank that a minimal action can achieve?
\end{question}

\begin{question}
    Are there any groups for which Theorem \ref{teorema1.1} does not hold?
\end{question}

For Corollary~\ref{cor:ex}, we constructed examples with $r_c(X, G) < \infty$ by decomposing the maximal equicontinuous factor into two components: a proximal factor and a finite-to-one factor. We wonder whether this is always the case.

\begin{question}
    Let $(X, G)$ be a minimal continuous with $r_c(X, G) < \infty$. Is it always possible to decompose the maximal equicontinuous factor into a proximal factor followed by a finite-to-one factor?
\end{question}

\bibliography{ref}
\bibliographystyle{plain}

\end{document}